\documentclass[11pt]{article}
\usepackage{amssymb,amsmath,latexsym,amscd,pb-diagram}

\newtheorem{example}[equation]{Example}
\newtheorem{theorem}[equation]{Theorem}

\newtheorem{lemma}[equation]{Lemma}

\newtheorem{remark}[equation]{Remark}
\numberwithin{equation}{section}

\newcommand{\pr}{\prime}

\newcommand{\wh}{\widehat}

\newcommand{\Ctd}{\hbox{\text{\rm Ctd}}}

\newcommand{\ol}{\overline}
\newcommand{\Z}{\mathbb{Z}}

\newcommand{\fg}{{\mathfrak g}}

\newcommand{\autfun}{{\bf Aut}}

\newcommand{\qed}{\hfill $\Box$}

\newcommand{\End}{\hbox{End}}

\newcommand{\C}{{\mathbb C}}
\newcommand{\cL}{{\mathcal L}}

\newcommand{\Aut}{{\rm Aut}}

\newcommand{\Hom}{{\rm Hom}}
\newcommand{\id}{{\rm id}}

\newcommand{\kalg}{{k\text{\rm -alg}}}

\newcommand\Spec{\text{\rm Spec}}

\newcommand\os{\overset}
\newcommand\us{\underset}
\newcommand\et{\text{\rm \'et}}

\newcommand\q{\quad}

\newcommand{\bsig}{{\pmb \sigma  }}

\newcommand{\Der}{{\rm Der}}
\newcommand{\ad}{{\rm ad}}

\title{Derivations of certain  algebras defined by
\'etale descent}

\author{
Arturo Pianzola$\hbox{\,}^{1,2}$\thanks{The author gratefully
acknowledges the support of the Natural Sciences and Engineering
Research Council of Canada, and of CONICET.} \vspace{0.3cm}\\
$\hbox{\ \,}^1${\small University of Alberta, Department of
Mathematical and Statistical Sciences,}\\{\small Edmonton, Alberta,
Canada T6G 2G1}\\
$\hbox{\ \,}^2${\small Instituto Argentino de Matem\'atica, Saavedra
15, (1083) Buenos Aires, Argentina.} \\{\small Email:\
a.pianzola@math.ualberta.ca}}

\date{}
\begin{document}
\maketitle
\begin{abstract}
 \noindent We give an explicit description of the Lie algebra of derivations
 for a class of infinite dimensional algebras which are given by \'etale
descent. The algebras under consideration are twisted forms of
 central algebras over rings, and include
 the multiloop algebras that appear in
 the construction of extended affine Lie algebras.
\end{abstract}

\vskip.25truein \section{Introduction} Many interesting  infinite
dimensional Lie algebras can be thought as being ``finite
dimensional'' when viewed, not as algebras over the given base
field, but rather as algebras over their centroids. From this point
of view, the algebras in question look like twisted forms of simpler
objects with which one is familiar. The quintessential example of
this type of behaviour is given by the affine Kac-Moody Lie
algebras.

An affine Kac-Moody Lie algebra $\cL$ (derived modulo its centre)
has centroid $R=\C[t^{\pm 1}],$ and there exists a unique finite
dimensional simple Lie algebra ${\fg}$ (whose type is called the
absolute type of $\cL)$ such that
$$
\cL\otimes_R R'\simeq (\fg \otimes_{\C} R)\otimes _R R'
$$
with $R\to R'$ faithfully flat and \'etale (one can in fact take
$R'=\C [t^{\pm 1/m}]$ for a suitable $m\ge 1).$ In other words, as
$R$-algebras, $\cL$ and $\fg \otimes_{\C}R$ are locally isomorphic for
the \'etale topology on $\Spec(R).$  Since ${\bf Aut}(\fg)$ is
smooth, Grothendieck's descent theory allows us to compute the
isomorphism  classes of such algebras by means of the pointed sets
$H^1_{\et}\big(R,\text{\bf Aut}(\fg)\big).$ In fact, as we vary $\fg$
over the nine Cartan-Killing types $A_\ell,B_\ell,\dots ,E_8$ we
obtain 16 classes in the resulting $H^1_{\et},$ and these correspond
precisely to the isomorphism classes of the affine algebras
\cite{P1}.

Extended affine Lie algebras (EALAs for short) are natural and
rather elegant ``higher nullity" analogues of the affine algebras
(see \cite{AABGP}, \cite{N1}, and \cite{N2} for details. For a
beautiful survey of the theory of EALAs, see \cite{N3}). A
reasonable understanding of how these algebras fit within the
cohomological language of forms is now beginning to emerge
(\cite{ABFP1}, \cite{ABFP2}, \cite{GP1}, \cite{GP2} and \cite{P1}).

Both the affine algebras and their EALA descendants have connections
with Physics, and it is here where central extensions play a crucial
role. In the affine case for example, it is not the complex Lie
algebra  $\cL$ that is of interest to physicists,  but rather its
one-dimensional universal central extension $\wh{\cL} = \cL \oplus
\C c.$ This presents an interesting duality: $\cL$ can be viewed as
a twisted form when thought as an algebra over $R=\C [t^{\pm1}]$, but
not when viewed as a complex Lie algebra. By contrast, as an
$R$-algebra, $\cL$ is centrally closed, but as a $\C$-algebra it is
not. The relevant central extension $\wh{\cL} = \cL\oplus \C c $
exists over $\C$ but not over $R.$\footnote{ In fact, $\wh{\cL} $ is
not even an $R$-algebra in any meaningful way.}

It is thus somehow surprising that natural central extensions of
twisted forms of Lie algebras can be obtained solely from their
defining descent data \cite{PPS}.\footnote{Some interesting
difficulties arise from the fact that  central extensions bring
cyclic homology into the picture, but there is no \'etale descent
for cyclic homology.} What is missing from this natural descent
construction of central extensions is a good understanding of when
they are universal.  This brings us to the current work of E.~Neher.

Just as the affine algebras are built out of loop algebras by adding central extensions and 
derivations, Neher has shown how to build
EALAs  out of basic objects called Lie tori \cite{N1}
\cite{N2}.\footnote{The multiloop algebras appear as the centreless
cores  of the EALAs that one is trying to build. They are Lie tori
as well as a very special type of twisted form. It is in this way
that the connection between EALAs and Galois cohomology emerges.}
Furthermore, in work in progress \cite{N4}, Neher has also shown how
to relate his construction of (universal) central extensions to the
one given by descent in \cite{PPS}. To do this, however, a
particular explicit description of the algebra of derivations of
multiloop algebras is needed.

The structure of the algebra of derivations of a multiloop algebra
has recently been determined by S. Azam \cite{A}. Azam's proof,
which is motivated by earlier work of Benkart and Moody \cite{BM},
is rather involved and depends on a delicate induction reasoning. In
this note we give an explicit description of the algebra of
derivations for a large class of algebras defined by \'etale
descent. Our methods are quite transparent and, when applied to the
particular case of multiloop algebras, yield a new concise and
conceptual proof of the Benkart-Moody-Azam result.

\section{Centroids of algebras and their derivations}

Throughout this note $k$ will denote a ring (commutative and unital)
and $k$-alg the category of associative commutative and unital
$k$-algebras.  Fix an object $R$ of $\kalg.$

Let $\cL$ be an $R$-algebra (not necessarily associative,
commutative or unital). Recall that the {\it centroid} of $\cL$
consists of the endomorphisms of the $R$-module $\cL$ that commute
with left and right multiplication by elements of $\cL.$ That is,
$$
\Ctd_R(\cL) =\{\chi  \in \,\End_R(\cL):\chi  (xy) =\chi  (x)y =
x\chi  (y) \, \forall x,y \in \cL.\}
$$
for all $x,y \in \cL.$ The centroid is a subalgebra of the (associative and unital)
$R$-algebra $\End_R(\cL).$  For each $r\in R$ the homothety $\chi
_r:x\mapsto rx$ belongs to $\Ctd_R(\cL).$  This yields an
$R$-algebra homomorphism
\begin{equation}\label{chiR}
\chi _{_{\cL, R}}:R\to \Ctd_R(\cL)
\end{equation}
which is injective if and only if $\cL$ is faithful. Recall that
$\cL$ is called {\it central} if the map $\chi _{_{\cL, R}}$ is an
isomorphism, and {\it perfect} if $\cL$ is spanned as a $k$-module
(in fact as an abelian group) by the set $\{xy:x,y\in \cL\}.$  By
restriction of scalars we can view $\cL$ also as a $k$-algebra.  At
the centroid level, this yields the (in general proper) inclusion
\begin{equation}\label{centroidinclusion}
\Ctd_R(\cL) \subset \;\Ctd_k(\cL).
\end{equation}
Perfectness, on the other hand, is independent of whether we view
$\cL$ as an algebra over $R$ or $k.$

 For convenience we recall the following simple yet useful
fact (see \cite{J}, \S4 of \cite{ABP2} and  \cite{BN} for details
and more general results on centroids).

\begin{lemma}\label{perfect}
If $\cL$ is perfect the centroid $\Ctd_R(\cL)$ is commutative and
the inclusion $\Ctd_R(\cL)\subset \;\Ctd_k(\cL)$ is an equality.
\qed
\end{lemma}

This will be the situation that will be considered in our work. In
particular $\Ctd_R(\cL)$ is an object of $k$-alg and $\cL$ can
naturally be viewed as an algebra over the (commutative) ring
$\Ctd_R(\cL).$
\bigskip

We finish this section by describing the main problem that we want
to study. By restriction of scalars we may view $\cL$ as a
$k$-algebra. We then have a natural $k$-Lie algebra homomorphism $
\eta_{\cL} : \Der_k(\cL)\to \Der_k\big(\Ctd_k(\cL)\big)$ given by
\begin{equation}\label{derivationmap}
\eta_{\cL}(\delta)(\chi  ) = [\delta,\chi ] = \delta \circ \chi
-\chi \circ \delta
\end{equation}
for all $\delta \in \Der_k(\cL)$ and $\chi \in \Ctd_k(\cL).$
 Assume,
furthermore, that  $\cL$ is such that the natural map
\begin{equation}\label{chik}
\chi_{\cL} : R \os {\chi _{_{\cL, R}}} \longrightarrow \Ctd_R(\cL)
\subset \Ctd_k(\cL)
\end{equation}
 obtained by composing (\ref{chiR}) and (\ref{centroidinclusion}) is
an isomorphism. Then $\eta_{\cL}$ induces a $k$-Lie algebra
homomorphism (also denoted $\eta_{\cL}$)
\begin{equation}\label{natural}
\eta_{\cL}  :\,\Der_k(\cL) \to \Der_k(R).
\end{equation}

For future reference let us observe
that the isomorphism $\chi_{\cL} : R \to \Ctd_k(\cL)$ under
consideration is given by $r\mapsto \chi _r$ where $\chi _r:x\mapsto
rx.$  Thus for $\delta \in \,\Der_k(\cL)$ our map (\ref{natural}) is
determined by the identity
\begin{equation}\label{meta}
\eta_{\cL} (\delta)(r) = t\Longleftrightarrow [\delta,\chi _r]
=\chi_t \, \, \text{for all} \, \, r,t \in R.
\end{equation}

One could say that the main objective of our work is to identify a
useful class of algebras for which the map $\eta_{\cL}$ is well
understood. In \S\ref{stwisted} we will discuss a class of algebras
for which the map $\chi_{\cL} : R \to\Ctd_k(\cL)$ is an isomorphism,
while in \S\ref{sdescent} we describe a class of algebras for which
the map $\eta_L$  has a natural section. This leads to an explicit
description of the algebra of derivations of the $k$-algebra $\cL.$
Finally, in \S\ref{smultiloop} we apply the results of the two
previous sections to study the case of multiloop algebras in detail.

\section{Twisted forms of algebras}\label{stwisted}

In what follows $A$ will denote an algebra (not necessarily
associative...) over $k.$ For each object $S$ in $\kalg$ we will
find it at times convenient to denote the resulting $S$-algebra $A
\otimes_k S$ by $A_S.$

\begin{lemma}\label{centroid}
Assume that the $k$-algebra $A$ is finitely presented as a $k$-module,
and that $k\to R$ is flat.  Then the canonical map
$$
\nu_{A, k, R}  :\,\Ctd_k(A)\otimes_k R \to \;\Ctd_R(A\otimes_k R) =
\Ctd_R(A_R)
$$
is an $R$-algebra isomorphism.\footnote{The assumptions we have made
on $A$ and $R$ are natural within the context of the present work,
but the main example that we have in mind is of course when $k$ is a
field and $A$ is finite dimensional. In the case when $k$ is a
field, many other examples when the map $\nu_{A,k,R}$ is an
isomorphism can be found in \cite{A}, \cite{ABP2.5} and \cite{BN}.}
\end{lemma}

\begin{proof}
The map in question is the restriction to $\Ctd_k(A)\otimes_k R$ of
the canonical map $ \End_k(A)\otimes_k R\to \;\End_R(A\otimes_k R) =
\End_R(A_R).$ Let
$$
\beta  = \beta  _{A,k}:\,\End_k(A) \to \;\Hom_k(A\otimes_k A,
A\oplus A)
$$
be the unique $k$-linear map satisfying
$$
\beta_{A,k}  (f)(a_1\otimes a_2) = \big(f(a_1a_2) -f (a_1)a_2,
f(a_1a_2) - a_1f(a_2)\big).
$$
By definition $\Ctd_k(A) = \ker(\beta  _{A,k}).$  We have the
commutative diagram

$$\scriptsize{
\begin{array}{cccccl}
0 \longrightarrow &\Ctd_k(A)\otimes_k R &\longrightarrow
&\End_k(A)\otimes_k R &\os{\beta _{A,k} \otimes 1}\longrightarrow
&\Hom_k(A\otimes _k
A,A\oplus A) \otimes_k R\\
&\downarrow\nu_{A, k, R}   &&\downarrow  &&\q\q\q\q\q\q\q \downarrow  \\
0 \longrightarrow &\Ctd_R(A_R) &\longrightarrow &\End_R(A_R)
&\os{\beta  _{A_R, R}}\longrightarrow &\Hom_R(A_R\otimes_R A_R, A_R
\oplus A_R).
\end{array} }
$$
The top row is exact because $k \to R$ is flat, the middle vertical
arrow is bijective because $A$ is finitely presented, while the
right vertical map is injective because $A \otimes_k A$ is of finite
type \cite[Ch.~1, \S2.10, Prop. 11]{Bbk}. From this it readily
follows that $\nu_{A,k,R}$ is an isomorphism. \qed

\bigskip

 The $k$-group
functor of automorphisms of $A$ will be denoted by $\autfun(A).$
Thus
$$
\begin{aligned}
\autfun(A): & \,\,\kalg \to \;{\rm Grp}\\
& \,\, S \mapsto \autfun(A)(S) = {\rm Aut}_{S\text{\rm -alg}}(A_S),
\end{aligned}
$$
where this last is the group of automorphisms of the $S$-algebra
$A_S.$ If as a $k$-module $A$ is projective of finite type, then
$\autfun(A)$ is an affine group scheme over $\Spec(k).$

Recall that  a {\it twisted form} of the $R$-algebra $A_R$ for the
fpqc topology on $\Spec(R)$ is an $R$-algebra $\cL$ such that
\begin{equation}\label{trivialize}
\cL \otimes_R R^\pr \simeq_{R^\pr\text{\rm -alg}}\;A_R\otimes _R
R^\pr
\end{equation}
for some faithfully flat  extension $R\to R^\pr$. Given a form $\cL$
as above for which (\ref{trivialize}) holds, we say that $\cL$ is
{\it trivialized} by $R^\pr.$  The $R$-isomorphism classes of such
algebras can be computed by means of cocycles, just as one does in
Galois cohomology \cite{Se}:
\begin{equation}\label{correspondence}
 \text{\it Isomorphism classes of}\, R^\pr/R\text{\it -forms of}\; A_R
\longleftrightarrow H^1\big(R'/R,\autfun(A)\big).
\end{equation}
 Since we will need the explicit
description of this correspondence for our work, we will briefly
recall the basic relevant facts.\footnote{ The general theory of
descent we are using can be found within \cite{SGA1} and
\cite{SGA3}. The formalism of torsors is clearly presented in
\cite{DG} and \cite{M}. An accessible account that is (almost)
sufficient for our needs can be found in \cite{KO} and \cite{Wth} .}
Let $R^{\pr\pr} = R^\pr \otimes _R R^\pr$ and $R^{\pr\pr\pr} = R^\pr
\otimes _R R^\pr \otimes_R R^\pr.$ We have the standard $R$-algebra
homomorphisms $p_i:R^\pr \to R^{\pr\pr},$ $i=1,2$ and $p_{ij}:
R^{\pr\pr}\to R^{\pr\pr\pr},$ $1\le i <j\le 3$ corresponding to the
projections on the $i$-th and $(i,j)$-th components respectively (see \cite{Wth} \S 17.6 for details).
These yield group homomorphisms (also denoted by $p_i$ and $p_{ij})$
$$
\begin{aligned}
\autfun (A)(R^\pr) &\os{p_i}\longrightarrow
\,\autfun(A)(R^{\pr\pr}),\\
\autfun(A)(R^{\pr\pr}) &\os{p_{ij}}\longrightarrow
\,\autfun(A)(R^{\pr\pr\pr}).
\end{aligned}
$$
An $R^\pr/R${\it -cocycle with values in} $\autfun (A)$ is an
element $u\in \,\autfun(A)(R^{\pr \pr})$ such that $p_{13}(u) =
p_{23}(u)p_{12}(u).$  If $g\in\,\autfun (A)(R^\pr)$ then $g \cdot u=
p_2(g)u p_1(g)^{-1}$ is also a cocycle. This defines an action of
the group $\autfun(A)(R^{\pr})$  on the set of cocycles, and we
define $H^1\big(R^\pr/R,\autfun(A)\big)$ to be the quotient set
(whose elements are thus equivalence classes of cocycles) defined by
this action. $H^1\big(R^\pr/R,\autfun(A)\big)$ is a pointed set
whose distinguished element is the class of the cocycle $1 \in
\autfun(A)(R^{\pr \pr}).$

The correspondence (\ref{correspondence}) is given by attaching to a
cocycle $u$ the $R$-algebra
$$
\cL_u =\{x\in A\otimes_k R^\pr:up^A_1(x) = p^A_2(x)\}
$$
where  $p^A_i = \id_A \otimes p_i : A \otimes_k  R' \to A \otimes_k
R' \otimes_R R'.$ If $\mu :R^\pr \otimes_R R^\pr\to R^\pr$ is the
map corresponding to the multiplication of the ring $R'$, then the
diagram
$$
\begin{array}{clc}
A\otimes_k R^\pr \otimes_R R^\pr\\
&\searrow^{ ^{1\otimes \mu}}  \\
\bigcup&&A\otimes_k R^\pr\\
&\nearrow\\
\cL_u \otimes_R R^\pr
\end{array}
$$
induces an isomorphism $\cL_u\otimes_R R^\pr \simeq A\otimes_k
R^\pr.$
\end{proof}

\begin{lemma}\label{basic}
Let $\cL$ be a twisted form of $A_R$ for the fpqc topology on
$\Spec(R).$ Assume  $A$ is perfect and central as a $k$-algebra, and
finitely presented as a $k$-module. Then

\begin{description}
\item  {\rm (1)} $\cL$ is perfect.  In particular $\Ctd_R(\cL)$ is
commutative and the inclusion $\Ctd_R(\cL) \subset \,\Ctd_k(\cL)$ is
an equality.
\item  {\rm (2)} As an $R$-module $\cL$ is faithful and finitely presented.
\item  {\rm (3)} The canonical map $ \chi_{\cL} : R \to \Ctd_k(\cL)$ is
an isomorphism.
\end{description}
\end{lemma}

\begin{proof}
(1) Since perfectness is preserved by base change $A_{R^\pr}$ is
perfect. A straightforward faithfully flat descent argument
\cite[Lemma 4.6.1]{GP2} yields that $\cL$ is perfect. The rest of
(1) now follows from Lemma \ref{perfect}.

(2) and (3) By Lemma \ref{centroid} the canonical map $ R^\pr \to
\,\Ctd_{R^\pr}(A\otimes_k R^\pr)$ is an isomorphism. By reasoning as
in \cite[Lemma 4.6.2,3]{GP2} we see that (2) holds, and also that
the canonical map $\chi_{_{\cL, R}} : R \to \Ctd_R(\cL)$ is an
isomorphism. Now (3) follows from (1). \qed
\medskip

Our next objective is to show that for a large class of twisted
forms (that include multiloop algebras) the Lie algebra homomorphism
$\eta_{\cL} $ defined in (\ref{natural}) admits a {\it natural}
section.  One of the crucial assumptions is that the faithfully flat
trivializing base change $R \to R'$ be \'etale.

\section{Derivations of certain algebras defined by \'etale
  descent}\label{sdescent}

Assume that $R \to R'$  is a faithfully flat and \'etale morphism in
$k$-alg. Let $d\in \,\Der_k(R).$  We view $d$ naturally as an
element of $\Der_k(R,R^\pr)$ via $R\to R^\pr.$ Since $R\to R^\pr$ is
faithfully flat we can (and henceforth will) naturally identify $R$
with a $k$-subalgebra of $R^\pr.$ After this identification the
assumption that $R\to R^\pr$ is \'etale yields the existence of a
unique $d^\pr \in\,\Der_k(R^\pr)$ extending $d$ \cite{EGAIV} Cor. 20.5.8. Similarly $d^\pr$
extends to two derivations of $R^{\pr\pr} = R^\pr \otimes_R R^\pr$
via the two morphisms $p_i:R^\pr {_\rightarrow\atop^\rightarrow}
R^{\pr\pr}.$ If we denote these by $d^{\pr\pr}_1$ and
$d^{\pr\pr}_2,$ then both $d^{\pr\pr}_1$ and $d^{\pr\pr}_2$ extend
$d$ under $R\to R^\pr {_\rightarrow\atop^\rightarrow} R^{\pr\pr}.$
Since these two composite maps coincide and the resulting map $R\to
R^{\pr\pr}$ is \'etale, it follows that $d^{\pr\pr}_1 = d^{\pr\pr} =
d^{\pr\pr}_2$ for some unique $d^{\pr\pr}\in\,\Der_k(R^{\pr\pr}).$
In particular $d^\pr(s)\otimes 1 = d^{\pr\pr}(s\otimes 1)$ and
$1\otimes d^\pr(s) =d^{\pr\pr}(1\otimes s)$ for all $s\in R^\pr.$
That is
\begin{equation}\label{useful1}
p_i \circ d^{\pr} = d^{\pr \pr} \circ p_i.
\end{equation} Let $R^{\pr\pr}_0 = \{s\in R^{\pr\pr} : d^{\pr\pr}(s) =
0\; \text{for all}\; d\in \,\Der_k(R)\}.$  It is clear that
$R^{\pr\pr}_0$ is a $k$-subalgebra of $R^{\pr\pr}.$ This yields a
group homomorphism $\autfun(A)(R^{\pr\pr}_0) \to
\autfun(A)(R^{\pr\pr})$ for any $k$-algebra $A.$ An element $ u \in
\autfun(A)(R^{\pr\pr})$ is in the image of this map if and only if
$u(A \otimes 1 \otimes 1) \subset A \otimes_k R^{\pr \pr}_0.$
\end{proof}

\begin{theorem}\label{main}
Let $A$ be a  $k$-algebra which is finitely presented as a
$k$-module. Let $R\to R^\pr$ be a faithfully flat and \'etale
extension in $k$-alg with $k \to R$ flat. Consider the twisted form
$\cL_u$  of $A_R$ corresponding to a cocycle $u \in
\autfun(A)(R^{\pr\pr}).$  Assume that the following two conditions
hold.
\smallskip

${\rm (i)}$ The canonical map $\chi_{\cL} : R \to \Ctd_k(\cL_u)$ is
an isomorphism.\footnote{For example if $A$ is perfect and central.}

${\rm (ii)}$ The cocycle $u$ belongs to the image of $\autfun
(A)(R^{\pr\pr}_0)$ in $\autfun (A)(R^{\pr\pr}).$
\smallskip

\noindent Then the Lie algebra homomorphism
$$
\eta_{\cL_u}  :\,\Der_k(\cL_u) \to \,\Der_k(R)
$$
described in {\rm  (\ref{natural})}admits a natural section $\rho.$
Furthermore
$$
\Der_k(\cL_u) =\,\Der_R(\cL_u)\rtimes \rho\big(\Der_k(R)\big).
$$
In particular, if $k \to R$ is \'etale then $\Der_k(\cL_u)
=\,\Der_R(\cL_u).$
\end{theorem}

\begin{proof}
Let $d\in \,\Der_k(R).$  The map $\id_A\otimes d^{\pr\pr}:A\otimes_k
R^{\pr\pr}\to A\otimes_k R^{\pr\pr}$ is clearly a derivation of
$A\otimes_k R^{\pr\pr}$ as a $k$-algebra.  The key point is that
\begin{equation}\label{commutes}
\id_A\otimes d^{\pr\pr}\, \, \text{\it commutes with the action
of}\, \, u.
\end{equation}
Indeed, if $x\in A$ and we write $u^{-1}(x\otimes 1\otimes 1) =\sum
x_i\otimes s_i$ for some $x_i \in A$ and $s_i\in R^{\pr\pr}_0,$ then
using that $u$ is $R^{\pr\pr}$-linear and $d^{\pr\pr}(s_i)=0$ we see
that for all $s\in R^{\pr\pr}$ we have
$$
\begin{aligned}
u \circ (\id_A\otimes d^{\pr\pr})\circ u^{-1}(x\otimes s)
&= u \circ (\id_A\otimes d^{\pr\pr})\Big(\sum x_i\otimes s_is\Big)\\
&= u  \Big(\sum x_i\otimes s_i d^{\pr\pr}(s)\Big) \\
&= x\otimes d^{\pr\pr}(s) = (\id_A\otimes d^{\pr\pr})(x\otimes s).
\end{aligned}
$$
From (\ref{useful1}) we obtain
\begin{equation}\label{useful2}
p^A_i\circ (\id_A\otimes d^\pr) = (\id_A\otimes d^{\pr\pr})\circ
p^A_i .
\end{equation}
From (\ref{commutes}) and (\ref{useful2}) it follows that
$\id_A\otimes d^\pr \in \,\Der_k(A\otimes_k R^\pr)$ stabilizes
$\cL_u.$  Indeed if $y\in \cL_u$ then
$$
\begin{aligned}
up^A_1\big((\id_A\otimes d^\pr)(y)\big) &= (u\circ (\id_A\otimes
d^{\pr\pr})\circ
p^A_1)(y) = \big((\id_A\otimes d^{\pr\pr})\circ u \circ p^A_1\big)(y)\\
&= (\id_A\otimes d^{\pr\pr})\big(p^A_2(y)\big) = p^A_2
\big((\id_A\otimes d^\pr)(y)\big).
\end{aligned}
$$
This shows that there exists a $k$-linear map $\rho: \,\Der_k(R)\to
\,\Der_k(\cL_u) $  given by
\begin{equation}\label{rho}
\rho:d \mapsto (\id_A\otimes d^\pr) \vert  _{\cL_u}.
\end{equation}
It is clear that $\rho$ is a Lie algebra homomorphism.  To verify
that $\rho$ is a section of $\eta = \eta_{\cL_u}  $ in
(\ref{natural}) we observe that for all $d \in \Der_k(R)$, $y = \sum
x_i \otimes s_i \in \cL_u$ and $r \in R$ we have
$$
[\id_A\otimes d^\pr , \chi  _r ]\Big(\sum x_i\otimes s_i\Big) =\sum
x_i\otimes d^\pr(r)s_i = d^\pr(r)y = d(r)y = \chi_{d(r)}(y).
$$
According to (\ref{meta}) this shows that $\eta  \big(\rho(d)\big) =
d$ as desired.

Since by assumption $\Ctd_k(\cL_u)$ consists of the homotheties
$\chi _r$ for $r\in R,$ it is clear from the definition of $\eta  $
that $\ker(\eta)  =\,\Der_R(\cL_u).$ This shows that $\Der_R(\cL_u)$
is an ideal of $\Der_k(\cL_u).$ If moreover $\delta \in\,\Der_k
(\cL_u),$ then $\delta -\rho\big(\eta (\delta )\big) \in
\,\ker(\eta) .$ Thus $\Der_k(\cL_u) =\,\Der_R(\cL_u) +
\rho\big(\Der_k(R)\big).$ That this sum is direct is easy to see.
This completes the proof of our result.\qed
\end{proof}
\medskip

Before making a few relevant observations and remarks pertaining to
this last result, we should point out that the assumption made on
the cocycle $u$ is quite restrictive and certainly not necessary for
the thesis of the Theorem to hold. One is indeed fortunate that many
interesting cases fall under this assumption.

\begin{remark}\label{rcanonical} {\rm
  It
is important to observe that the definition of the section $\rho$
given by the Theorem, and the resulting identification of
$\Der_k(R)$ with a subalgebra of $\Der_k(\cL_u),$  are completely
natural and explicit: $d \in \Der_k(R)$ extends uniquely to a
derivation $d' \in \Der_k(R').$ Then $\id_A \otimes d'$ is a
derivation of the $k$-algebra $A \otimes_k R'$ and $\rho(d)$ is
nothing but the restriction of $\id_A\otimes d'$ to $\cL_u \subset A
\otimes_k R'.$

Note also that the isomorphism $R \simeq \Ctd_k (\cL_u)$ is given by
restricting the scalar action of $R$ in $A \otimes_k R'$ to $\cL_u.$
With this identification $\Ctd_k(\cL_u) = \{\chi_r : r \in R \}$,
and the natural action of $\Der_k(R)$ on $\Ctd_k (\cL_u)$ is then
given by $d(\chi_r) = \chi_{d(r)}.$}
\end{remark}

\begin{remark}\label{constant} {\rm The most natural type of twisted forms to which the Theorem applies
are those given by {\it constant} cocycles, namely when $u$ belongs
to the image of the map $\autfun(A)(k) \to \autfun(A)(R^{\pr \pr}).$
This is the case of multiloop algebras, as we will see in the next
section. }
\end{remark}

\begin{remark}\label{rcanonicalbis} {\rm Let $g \in \autfun(A)(R'),$ and set $v = p_2(g)up_1(g)^{-1}.$
Since $\cL_u \simeq \cL_v$ as $R$-algebras, the Lie algebra
homomorphism $\eta_{\cL_v} : \Der_k(\cL_v) \to \Der_k(R)$
corresponding to $\cL_v$ admits a section, but the cocycle $v$ need
not satisfy the assumption of the Theorem. An easy calculation shows
that
$$
\Der_k(\cL_u) = {\Der}_R(\cL_u) \rtimes \rho_g\big(\Der_k(R) \big)
$$
where  $\rho_g$ is given by $\rho_g(d) = g(\id_A \otimes d')g^{-1}$
for all $d \in \Der_k(R).$}
\end{remark}

\begin{example} \label{Lie}
{\rm Assume that $k$  is a field of characteristic $0,$ and that $A$
is a finite dimensional central simple Lie algebra over $k.$  To
abide by standard notational conventions we will denote $A$ by
${\fg}.$

Every derivation of the $R$-Lie algebra ${\fg}_R$ is inner. This
follows from theorem 1.1 of \cite{BM}, and also by the following
direct reasoning: If $\delta \in \,\Der_R({\fg}_R)$ we may view the
restriction of $\delta$  to $\fg$ as a derivation $\delta_{\fg}$ of
$\fg$
 with values in $\fg_R$ (via the adjoint representation). Since $\fg$ is finite dimensional
 there exists a finite dimensional submodule $M$ of $\fg_R$ such that $\delta_{\fg}$ takes values in $M.$
 By Whitehead's lemma there exists $x \in M \subset {\fg}_R$
such that $\delta_{\fg}  (y) = [x,y].$  This shows that $\delta  $
and $\ad_{\fg_R} (x)$ agree on ${\fg}\otimes 1.$ By $R$-linearity
$\delta = \ad_{\fg_R} (x)$.

Assume now that $\cL $ is a twisted form of ${\fg}_R.$ Choose a
faithfully flat base change $R \to R'$ that trivializes $\cL.$ Since
$\cL \otimes_R R' \simeq \fg \otimes_k R'$ and every derivation of
$\fg \otimes_k R'$ is inner, we conclude that $\big(\Der_R(\cL)/{\rm
IDer}(\cL)\big) \otimes_R R' = 0.$  By faithfully flat descent ${\rm
IDer}(\cL) = \Der_R(\cL).$

Finally, if $\cL =\cL_u$ for a cocycle $u$ as in  Theorem \ref{main} then
$$
\Der_k(\cL_u) = {\rm I\Der}(\cL_u) \rtimes \rho\big(\Der_k(R)\big).
$$
}
\end{example}

\begin{example}\label{RR} {\rm Assume $R = k[t]$ where $k$ is a
field. Let $A$ be an algebra as in  Theorem \ref{main}, and assume
that the connected component of the identity of the algebraic group
$\autfun(A)$ is reductive. Let $\cL$ be a twisted form of $A_R$
which is trivialized by the base change $k[t] \to k_s[t],$ where
$k_s$ is the separable closure of $k.$\footnote{By a theorem of
Steinberg this assumption is superfluous if $k$ is of characteristic
$0.$} From the work of Raghunathan and Ramanathan \cite{RR} we know
that the natural map $H_{\et }^1\big(k, \autfun(A)\big) \to
H_{\et}^1\big(k[t], \autfun(A)\big)$ is bijective. Thus  $\cL \simeq
\cL_u$ as $R$-algebras (a fortiori also as $k$-algebras) for some
constant cocycle $u.$ We have
$$
\Der_k(\cL) \simeq {\Der}_k(\cL_u) = {\Der}_R(\cL_u) \rtimes
\rho\big(R \frac{d}{dt} \big).
$$
The action of $R \frac{d}{dt} = \Der_k(R)$ on $\cL,$ however, is not explicit.
Of course if $\cL$ is given to us in the form $\cL =\cL_v,$ then we
can apply the considerations described in Remark
\ref{rcanonicalbis}.

For the Laurent polynomial ring $k[t^{\pm 1}]$ the situation is much
more delicate. The natural map $H_{\et}^1\big(k[t^{\pm 1}],
\autfun(A)\big) \to H_{\et}^1\big(k((t)), \autfun(A)\big)$ is
bijective whenever $\autfun(A)$ is reductive and the characteristic
of $k$ is good \cite{CGP}. If, for example, the image of the class
of $u$ in $H_{\et}^1\big(k((t)), \autfun(A)\big)$ is constant (a
problem that in theory can be studied by Bruhat-Tits methods), then
Theorem \ref{main} can be applied.}
\end{example}

\section{The Galois case. Applications to multiloop
algebras}\label{smultiloop}
 Throughout this section $k$ is assumed to be a field, and $A$ will denote a $k$-algebra that satisfies
 assumption (i) of  Theorem \ref{main}.\footnote{For example $A$ finite dimensional and central simple.}
 In this situation the canonical maps $\autfun
 (A)(k) \to \autfun (A)(R^{\pr\pr}_0) \to \autfun (A)(R^{\pr\pr})$
 are all injective, and we identify the first two groups with their
 respective images. We denote $\autfun(A)(k)$ by ${\rm Aut}_k(A).$

 Assume  that our form $\cL_u$ is trivialized by a (finite) Galois
extension $R^\pr$ of $R$  (see \cite{KO}, \cite{Wth} or, ultimately
and inevitably,
 \cite{SGA1}). Recall then that $R \to R'$ is faithfully flat, and that if $\Gamma \subset {\rm
Aut}_R(R^\pr)$ denotes the Galois group of the extension then the
map
$$
R^\pr \otimes _R R^\pr \to R^\pr \times \dots \times R^\pr \q (\vert
\Gamma  \vert  \text{\rm -times)}
$$
given by
$$
 a\otimes b \mapsto \big(\gamma  (a)b\big)_{\gamma
\in \Gamma }
$$
is an isomorphism of $R'$-algebras (with $R'$ acting by
multiplication on the second component of $R' \otimes_R R'$). Under
the resulting identification of $\autfun (A)(R^{\pr\pr})$ with $
\autfun (A)(R^\pr)\times\dots \times \,\autfun(A)(R^\pr)$ our
cocycle $u$ corresponds to a $\vert \Gamma  \vert  $-tuple
$(u_\gamma )_{\gamma \in\Gamma  }$ which satisfies the usual cocycle
condition $u_{\gamma \mu}   = u_\gamma \,^\gamma u_\mu,  $ with
$\Gamma  $ acting on $\autfun(A)(R^\pr)$ in the natural way.

\begin{remark}\label{cocycles} {\rm At the
level of Galois cocycles the assumption that $u$ be an element of
$\autfun (A)(R^{\pr\pr}_0)$ translates into the following condition:
For all $\gamma \in \Gamma$ we have $u_\gamma \in
\autfun(A)(R^\pr_0)$ where
 $R^\pr_0 =\{s\in R^\pr :d^\pr(s) =0
\,\, \text{for all} \,\,  d \in \,\Der_k(R)\}.$

Note that this condition is automatically satisfied whenever the
$u_\gamma $ are obtained from automorphisms of the $k$-algebra $A,$
i.e. $u_\gamma  =v_\gamma  \otimes 1$ for some $v_\gamma \in {\rm
Aut}_k(A).$ The action of $\Gamma  $ is in this case trivial, and
the cocycle condition simply states that $\gamma \mapsto v_\gamma $
is a group homomorphism from $\Gamma  $ to ${\rm Aut}_k(A).$  This
situation arises in the case of multiloop algebras, as we now
explain.}
\end{remark}

We will assume henceforth that $R= k[t^{\pm 1}_1,\dots,t^{\pm
1}_n].$ Fix an $n$-tuple ${\bf m} = (m_1,\dots,m_n)$ of positive
integers, and set $R_{\bf m} = R^\pr = k[t^{\pm
\frac{1}{m_1}}_1,\dots,t^{\pm \frac{1}{m_n}}_n].$ We assume in what
follows that the $m_i$ are relatively prime to the characteristic of
$k.$ The natural map $R\to R^\pr$ is then faithfully flat and
\'etale.

Let $m= \Pi_{1 \leq i \leq n} m_i.$  Assume $k$ contains a primitive
$m$-th root of unity $\xi_m ,$ and set  $\xi _{m_i} = \xi_m
^{\Pi_{j\ne i}m_j}.$ Then $R\to R^\pr$ is Galois with Galois group
$\Gamma = \Z/m_1\Z \times \cdots \times \Z/m_n\Z,$ where for each
 ${\bf e} = (e_1,\dots ,e_n)\in \Z^n$ the corresponding element $\ol{\bf e} = (\ol e_1,\cdots,\ol e_n) \in \Gamma  $ acts on $R'$
via
$$
^{\ol {\bf e}} t^{\frac{1}{m_i}}_i = \xi  ^{e_i}_{m_i}
t^{\frac{1}{m_i}}_i.
$$

Multiloop algebras arise under these assumptions by considering an
$n$-tuple $\bsig = (\sigma  _1,\dots,\sigma  _n)$ of commuting
elements of ${\rm Aut}_k(A)$ satisfying $\sigma  ^{m_i}_i = 1.$  For
each $n$-tuple $(i_1,\dots ,i_n)\in \Z^n$ we consider the
simultaneous eigenspace
$$
A_{i_1 \dots i_n} =\{x\in A:\sigma  _j(x) = \xi  ^{i_j}_{m_j} x \,
\, \text{\rm for all} \,\, 1\le j\le n\}.
$$
Then $A = \sum A_{i_1 \dots i_n}, $ and $A = \bigoplus A_{i_1 \dots
i_n}$ if we restrict the sum to those $n$-tuples $(i_1,\dots ,i_n)$
for which  $0 \leq i_j < m_j.$

The map $\Gamma \to {\rm Aut}_k(A)$ given by
$$
\ol{\bf e} = (\bar e_1,\dots,\bar e_n) \mapsto \sigma ^{-e_1}_1\dots
\sigma ^{-e_n}_n = v_{\bar {\bf e}}
$$
is a group homomorphism whose corresponding cocycle $u=(u_{\ol{\bf
e}})_{\ol{\bf e}\in \Gamma  }$ with $u_{\ol{\bf e}} = v_{\ol{\bf e}}
\otimes 1$ is constant (see Remark \ref{constant}), hence satisfies
assumption (ii) of Theorem \ref{main}. The corresponding form
$\cL_u$ is the {\it multiloop algebra} commonly denoted by
$L(A,\bsig),$
$$
\cL_u = L(A,\bsig) = \us{(i_1,\dots ,i_n)\in \Z^n}\oplus\,
A_{i_1\dots i_n} \otimes t^{\frac{i_1}{m_1}} \dots
t^{\frac{i_n}{m_n}}_n \subset A\otimes _k R_{\bf m}.
$$
We have
 \begin{equation}\label{multiloopder}
 \Der_k \big(L(A,\bsig)\big) = \Der_R\big(L(A,\bsig)\big)
\rtimes \rho\big(\Der_k (R)\big). \end{equation}

 We have
$\Der_k(R)  =  R\frac{\partial}{\partial t_1} \oplus .... \oplus R\frac{\partial}{\partial t_n},$ and  the (unique) way in which the elements of
$\Der_k(R)$ extend to $\Der_k(R_{\bf m})$ is clear. The
explicit action of $\Der_k (R)$ on $L(A,\bsig)$ and on
$\Ctd_k\big(L(A, \bsig)\big)$ is now as described in Remark
\ref{rcanonical}.

Finally, if $k$ is algebraically closed  of characteristic $0$ and
$A=\fg$ is a simple finite dimensional Lie algebra,  then the
multiloop algebras $L(\fg,\bsig)$ arise naturally in modern infinite
dimensional Lie theory as we have explained in the Introduction (for
example, if $n=1,$ the $L(\fg,\bsig )$ are the derived algebras of
the affine Kac-Moody Lie algebras, modulo their centres). By Remark
\ref{Lie}
\begin{equation}\label{Liemultiloopder}
\Der_k\big(L(\fg,\bsig)\big) = {\rm IDer}\big(L(\fg,\bsig
)\big)\rtimes \rho\big(\Der_k (R)\big).
\end{equation}

\begin{remark}
{\rm The analogue of Theorem \ref{main} for automorphisms instead of
derivations {\it fails}.  We do have an exact sequence
$$
1 \to \,\Aut_R(\cL_u) \to \,\Aut_k (\cL_u) \os \eta  \to \,\Aut_k(R)
$$
for all cocycles $u.$  The homomorphism $\eta  $ need not be
surjective, and even when it is, and if $u$ satisfies the assumption
of the Theorem, the resulting exact sequence need not split.  This
situation takes place, for example, when  $A= M_{2\times 2}(\C), $
$R = \C [t^{\pm1}_1,t^{\pm1}_2],$ and $\cL_u$ is the standard
quaternion algebra over $R$ (see {\rm \cite{GP2}} example 4.11)}.
\end{remark}
\medskip

\noindent {\bf Acknowledgement}  I would like to express my gratitude to Erhard Neher for his many useful  comments.

\end{document}